\documentclass[11pt,a4paper]{amsart}

\usepackage{a4wide}

\usepackage[english]{babel}

\usepackage[pdftex]{graphicx}

\usepackage[all]{xy}
\usepackage{amsmath,amsthm,amssymb,enumerate}
\usepackage{latexsym}
\usepackage{amscd}

\usepackage[colorlinks=true]{hyperref}

\newtheorem{theorem}{Theorem}[section]
\newtheorem{thm}[theorem]{Theorem}
\newtheorem{definition}[theorem]{Definition}
\newtheorem{defn}[theorem]{Definition}
\newtheorem{prop}[theorem]{Proposition}

\newtheorem{cor}[theorem]{Corollary}

\newtheorem{remark}[theorem]{Remark}
\newtheorem{rem}[theorem]{Remark}
\newtheorem{example}[theorem]{Example}

\newtheorem{question}[theorem]{Question}

\newcommand{\cat}{\mathcal}

\newcommand{\R}{\mathbb R}
\newcommand{\Q}{\mathbb Q}
\newcommand{\Z}{\mathbb Z}
\newcommand{\N}{\mathbb N}
\newcommand{\C}{\mathbb C}
\newcommand{\T}{\mathbb T}
\newcommand{\K}{\mathbb K}
\newcommand{\CP}{\mathbb P}

\newcommand{\Xx}{{\cat X}}

\newcommand{\Nn}{{\cat N}}

\newcommand{\om}{{\omega}}

\newcommand{\rp}{{{\R\CP}\,\!}}

\DeclareMathOperator{\rank}{rank}

\DeclareMathOperator{\vol}{vol}
\DeclareMathOperator{\conv}{conv}

\DeclareMathOperator{\Aff}{Aff}

\DeclareMathOperator{\Chi}{{\displaystyle{\chi}}}

\def\om{{\omega}}

\def\sec{{\mathfrak s}}
\def\rs{{\mu_\text{RS}}}

\begin{document}
\title{On the mean Euler characteristic of Gorenstein toric contact manifolds}

\author[M.~Abreu]{Miguel Abreu}
\address{Center for Mathematical Analysis, Geometry and Dynamical Systems,
Instituto Superior T\'ecnico, Universidade de Lisboa, 
Av. Rovisco Pais, 1049-001 Lisboa, Portugal}
\email{mabreu@math.tecnico.ulisboa.pt}
 
\author[L.~Macarini]{Leonardo Macarini}
\address{Universidade Federal do Rio de Janeiro, Instituto de Matem\'atica,
Cidade Universit\'aria, CEP 21941-909 - Rio de Janeiro - Brazil}
\email{leomacarini@gmail.com}

\thanks{MA was partially funded by FCT/Portugal through UID/MAT/04459/2013 and 
project PTDC/MAT-GEO/1608/2014, and by the Institut Mittag-Leffler, Sweden, and 
CNPq/Brazil through visiting grants. LM was partially supported by CNPq/Brazil and by 
FCT/Portugal through a visiting grant. The present work is part of the authors activities 
within BREUDS, a research partnership between European and Brazilian research 
groups in dynamical systems, supported by an FP7 International Research Staff Exchange 
Scheme (IRSES) grant of the European Union.}

\date{March 13, 2018}

\begin{abstract}
We prove that the mean Euler characteristic of a Gorenstein toric contact manifold, i.e. a good
toric contact manifold with zero first Chern class, is equal to half the normalized volume of the 
corresponding toric diagram and give some applications. A particularly interesting one, obtained 
using a result of Batyrev and Dais, is the following: twice the mean Euler characteristic of a 
Gorenstein toric contact manifold is equal to the Euler characteristic of any crepant toric symplectic 
filling, i.e. any toric symplectic filling with zero first Chern class.
\end{abstract}



\maketitle

\section{Introduction}
\label{s:intro}

Good toric contact manifolds are the odd dimensional analogues of closed toric symplectic
manifolds. As proved by Lerman in~\cite{Le1}, they can be classified by the associated moment 
cones, in the same way that Delzant's theorem classifies closed toric symplectic manifolds by 
the associated moment polytopes. 

In~\cite{AM} we showed that on any good toric contact manifold any non-degenerate toric contact 
form is even, i.e. all contractible closed orbits of its Reeb flow have even contact homology degree. 
As we will see in Proposition~\ref{prop:parity}, this is also true for the non-contractible closed Reeb 
orbits. The contact homology degree of a non-degenerate closed Reeb orbit is determined by an 
appropriate dimensional shift of its Conley-Zehnder index. More precisely, it is equal to the 
Conley-Zehnder index plus $n-2$, where throughout this paper contact manifolds have dimension 
$2n+1$. The Conley-Zehnder index is a well defined integer provided the first Chern class of 
the contact structure vanishes.

\begin{defn} \label{def:Gorenstein}
A \emph{Gorenstein toric contact manifold} $(M,\xi)$ is a good closed toric contact manifold with
zero first Chern class. For $\alpha$ a non-degenerate toric contact form on $(M,\xi)$ with 
corresponding Reeb vector field $R_\alpha$, the \emph{contact Betti numbers} $cb_j (M, \alpha)$, $j\in\Z$,
are defined by
\[
cb_j (M, \alpha) = \text{number of closed $R_\alpha$-orbits with contact homology degree $j$}
\]
and the \emph{mean Euler characteristic} $\Chi (M, \alpha)$ is defined by
\[
\Chi (M, \alpha) := \lim_{N\to\infty} \frac{1}{N} \sum_{j=-N}^N (-1)^j cb_{j} (M,\alpha) \,.
\]
\end{defn}
\begin{rem} \label{rem:Gorenstein}
The name Gorenstein is motivated by the fact that, in Algebraic Geometry, the
corresponding toric cones are also called Gorenstein toric isolated singularities,
the singular point being at the apex of the cone. Hence, good toric contact manifolds
with zero first Chern class are links of Gorenstein toric isolated singularities.
\end{rem}
\begin{rem} \label{rem:meanEuler}
As mentioned above, on a Gorenstein toric contact manifold $(M,\xi)$ any non-degenerate toric
contact form $\alpha$ is even. Moreover, cf. Section~\ref{s:proof}, it has only finitely many simple 
closed Reeb orbits and the $\inf \{j\in\Z \mid cb_j (M,\alpha) \ne 0\}$ is finite (cf. Remark~\ref{rem:positive}). 
Hence, we have that
\[
\Chi (M,\alpha) = \lim_{N\to\infty} \frac{1}{2N} \sum_{j=0}^N \dim cb_{2j} (M,\alpha)
\]
and this limit always exists (cf. \cite{GG}).
\end{rem}

\begin{rem} \label{rem:invariance1}
Each contact Betti number $cb_j (M, \alpha)$ should be a contact invariant of $(M,\xi)$, 
the rank of its degree $j$ cylindrical contact homology, and $\Chi (M, \alpha)$ should then also be a 
contact invariant of $(M,\xi)$, its mean Euler characteristic $\Chi (M, \xi)$ as defined in~\cite{vK}.
Unfortunately, and despite recent foundational developments (e.g.~\cite{P1,P2}), cylindrical contact 
homology has not been proved to be a well defined invariant in the presence of contractible closed
Reeb orbits, even in this restricted context of Gorenstein toric contact manifolds. 

However, for Gorenstein toric contact manifolds that have crepant (i.e. with zero first Chern class)  
toric symplectic fillings, which is the context of Theorem~\ref{thm:main3} below, we can use positive
equivariant symplectic homology to conclude that the
contact Betti numbers are indeed true contact invariants. This follows 
from the recent work of McLean-Ritter~\cite{MR}, which uses previous work by  
Kwon and van Koert~\cite{KvK}, as we will now briefly explain.

Let $W_0$ be a crepant toric symplectic filling of $(M,\xi)$ and consider its completion
\[
W:= W_0 \cup (M \times [1,\infty), d(r\alpha))\,, 
\]
where $\alpha$ is a non-degenerate toric contact form on $(M, \xi)$, with $R_\alpha$ as its Reeb 
vector field, and $r$ is the coordinate on $[1,\infty)$. Then $W$ is a convex symplectic manifold to 
which the methods and results of Appendix D and Appendix E of~\cite{MR} can be applied. This is 
not an exact symplectic manifold in general (crepant toric symplectic fillings typically have embedded 
symplectic $2$-spheres corresponding to compact edges of their moment map images) but, using a 
new filtration on the symplectic chain complex constructed in \cite{MR}, 
one can define its positive equivariant symplectic (co)homology $ESH^*_+(W)$ and consider its corresponding
Morse-Bott spectral sequence. The first page of this spectral sequence is computed in Appendix E of \cite{MR}
and in our context we have that:
\begin{itemize}
\item Each Morse-Bott submanifold $B_\tau \subset M$ of initial points of the closed $R_\alpha$-orbits of
period $\tau$ is a disjoint union of circles $B_{\tau,c}$, i.e. $B_\tau = \cup_c B_{\tau,c}$.
\item Lemma 7.1 of \cite{MR} applies since $R_\alpha$ is a toric Reeb vector field.
\item The $S^1$-equivariant (co)homology of each $B_{\tau,c}$ is the (co)homology of a point, hence 
non-trivial only in degree zero.
\item The $E_1$ page of the third spectral sequence of Corollary 7.2 of \cite{MR} lives in total degree $k$ that 
satisfies
\[
k = 2n-1 - \text{(contact homology degree of some closed $R_\alpha$-orbit).}
\]
(Notice the difference between our contact homology grading and the one used in \cite{MR}.)
Since the contact homology degree is even, $k$ is odd and the spectral sequence collapses in this page.
\item It follows that each contact Betti number $cb_j (M, \alpha)$ is the rank of $ESH_+^{2n-1-j} (W)$ as
a $\K$-vector space, where $\K$ is the (universal) Novikov field.
\item This shows that the rank of each $ESH_+^{2n-1-j} (W)$ does not depend on the filling $W_0$ and
therefore
\[
cb_j (M, \xi) := {\rm rank}_\K ESH_+^{2n-1-j} (W) = cb_j (M, \alpha)
\]
is a well-defined contact invariant for Gorenstein toric contact manifolds $(M, \xi)$ that have crepant toric 
symplectic fillings.
\end{itemize}
It follows that if $\alpha$ is any non-degenerate toric contact form on $(M,\xi)$ then
\begin{align*}
\chi(M,\xi) & := \lim_{N\to\infty} \frac{1}{2N} \sum_{j=0}^{N} {\rm rank}_\K ESH_+^{2n-1-2j} (W) = \chi(M,\alpha)
\end{align*}
is also a well-defined contact invariant for Gorenstein toric contact manifolds $(M, \xi)$ that admit crepant toric symplectic fillings.

Although it is well-known that all Gorenstein toric contact manifolds of dimensions $3$ and $5$ have crepant 
toric symplectic fillings, that is no longer the case in higher dimensions (see Section~\ref{s:resolutions}).
\end{rem}

\begin{rem} \label{rem:invariance2}
Note that Theorem~\ref{thm:main2} below does show that $\Chi (M, \alpha)$ depends only on the 
Gorenstein toric contact manifold $(M,\xi)$, via the volume of its toric diagram (see below), and not on the particular 
non-degenerate toric contact form $\alpha$ used to define it. Hence, for any Gorenstein toric contact manifold
$(M,\xi)$ (without assuming the existence of a crepant filling), we can still define the mean Euler characteristic $\Chi (M, \xi)$ as
\[
\Chi (M, \xi) := \Chi (M, \alpha)\,,\ \alpha = \text{non-degenerate toric contact form on $(M,\xi)$}
\]
and consider it as a {\bf toric} contact invariant. In joint work in preparation with Miguel Moreira, we will show that each 
contact Betti number $cb_j (M, \alpha)$ is also independent of $\alpha$ and can be combinatorically determined 
using the Erhart polynomial of the toric diagram of $(M,\xi)$.
\end{rem}

\subsection{Main Result} \label{ss:main}

Any $(2n+1)$-dimensional good toric contact manifold is completely determined
by a good cone $C$ in the dual of the Lie algebra of the $(n+1)$-dimensional acting torus 
(see Definition~\ref{def:good} and Theorem~\ref{thm:good}).
For a Gorenstein toric contact manifold, there exists a $\Z$-basis of that torus such that the 
defining normals of $C\subset\R^{n+1}$ are of the form (see Corollary~\ref{cor:c_1})
\[
\nu_j = (v_j, 1)\,,\ v_j \in \Z^n\,,\ j=1,\ldots,d\,.
\]
The integral simplicial convex polytope $D = \conv (v_1, \ldots, v_d) \subset \R^n$ is called 
a toric diagram (see Definition~\ref{def:diagram}) and to any such diagram there corresponds 
a unique Gorenstein toric contact manifold $(M_D, \xi_D)$ of dimension $2n+1$ 
(see Theorem~\ref{thm:diagram}). 

Any vector $\nu = (v,1)\in\R^{n+1}$, with $v$ in the interior of $D$, determines a suitably 
normalized toric contact form $\alpha_\nu$ and corresponding toric Reeb vector $R_\nu$ 
(see Subsection~\ref{ss:reeb}). Moreover, each facet $F_\ell$ of $D$ determines a particular
simple closed $R_\nu$-orbit $\gamma_\ell$. Denote by $v_{\ell_1}, \ldots, v_{\ell_n} \in \{v_1, \ldots, v_d\}$ 
the vertices of $F_\ell$. Note that any facet of $D$ is an $(n-1)$-dimensional integral simplex in $\R^n$, 
whose only integral points are its $n$ vertices.

Our main result is the following formula for the mean index
\[
\Delta (\gamma_\ell) := \lim_{N\to\infty} \frac{\mu_{\rm RS} (\gamma_\ell^N)}{N}\,,
\]
where $\mu_{\rm RS}$ denotes the Robbin-Salamon index (see Section \ref{s:trivialization} 
for a discussion about this index and the trivialization of the contact structure).

\begin{thm} \label{thm:main}
\[
\frac{1}{\Delta (\gamma_\ell)} = \frac{n! \vol ( S_{v,\ell})}{2}\,,
\]
where $S_{v,\ell} \subset \R^n$ denotes the convex hull of $\{v, v_{\ell_1}, \ldots, v_{\ell_n}\}$.
\end{thm}

Since $S_{v,1}\,\ldots,S_{v,m}$, where $m$ is the number of facets of $D$, give a (simplicial)
subdivision of $D$ (see Figure~\ref{fig0}), we have that
\[
\sum_{\ell=1}^m \frac{1}{\Delta (\gamma_\ell)} = \frac{n! \vol (D)}{2}\,.
\]
\begin{figure}[ht]
\includegraphics{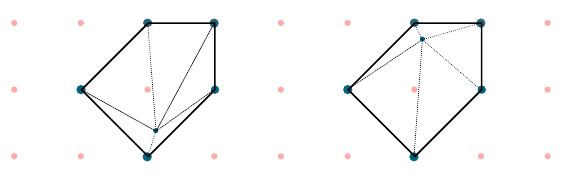}
\centering
\caption{Two toric Reeb vectors and corresponding (simplicial) subdivisions of a toric diagram.}
\label{fig0}
\end{figure}

Now, take $v = (r_1, \ldots, r_n)$ with $r_j$'s irrational and $\Q$-independent. Notice that this 
is a generic condition. Then $\alpha_\nu$ is non-degenerate and $\gamma_1,\ldots,\gamma_\ell$
are its only simple closed Reeb orbits (see Section~\ref{s:proof}). By applying the resonance relation 
\[
\sum_{\ell=1}^m \frac{1}{\Delta (\gamma_\ell)} =\Chi (M_D,\alpha_\nu)
\]
proved by Ginzburg and Kerman in~\cite{GK}, one gets the following very geometric formula
for the mean Euler characteristic of a Gorenstein toric contact manifold.
\begin{thm} \label{thm:main2}
Let $D\subset \R^n$ be a toric diagram and $(M_D,\xi_D)$ its corresponding Gorenstein
toric contact manifold. Then, for any non-degenerate toric contact form $\alpha_\nu$, with
$\nu$ in the interior of $D$, we have that
\[
\Chi (M_D,\alpha_\nu) = \frac{n! \vol (D)}{2}\,.
\]
\end{thm}
Taking into account Remark~\ref{rem:invariance1}, we get the following corollary.
\begin{cor} \label{cor:main2}
Let $D\subset \R^n$ be a toric diagram and $(M_D,\xi_D)$ its corresponding Gorenstein
toric contact manifold. If $(M_D,\xi_D)$ has a crepant toric symplectic filling, then 
\[
\Chi (M_D,\xi_D) = \frac{n! \vol (D)}{2}\,.
\]
In particular, $\vol (D)$ is a contact invariant of $(M_D,\xi_D)$.
\end{cor}

\subsection{Applications} \label{ss:app}

Using well-known facts about integral polytopes and their volumes, one gets the following
immediate applications of Theorem~\ref{thm:main2}.
\begin{cor} \label{cor:app1}
The mean Euler characteristic of a Gorenstein toric contact manifold is a half-integer.
\end{cor}
\begin{proof}
The normalized volume of an integral polytope $P\subset\R^n$,  i.e. $n! \vol (P)$, is always an integer.
\end{proof}
\begin{cor} \label{cor:app2}
In any given odd dimension and for any fixed upper bound, there are only finitely many
Gorenstein toric contact manifolds with bounded mean Euler characteristic.
\end{cor}
\begin{proof}
In any given dimension there are only finitely many integral polytopes with a fixed volume upper
bound.
\end{proof}
\begin{rem} \label{rem:apps}
These corollaries reflect the very special nature of toric contact structures. For example,
one might want to compare them with the following result of Kwon and van Koert~\cite{KvK}: 
given any rational number $x$, there is a Stein fillable contact structure $\xi_x$ on the 
$5$-sphere whose mean Euler characteristic equals $x$.
\end{rem}

Batyrev-Dais show in~\cite[Corollary 4.6]{BD} that the Euler characteristic of any crepant (i.e. 
with zero first Chern class) toric smooth resolution of a Gorenstein toric isolated singularity is 
equal to the normalized volume of the corresponding toric diagram. Since any crepant toric 
symplectic filling of a Gorenstein toric contact manifold gives rise to a crepant toric smooth 
resolution of the corresponding Gorenstein toric isolated singularity, we have the following very 
interesting application of Corollary~\ref{cor:main2}. It provides 
a relation between invariants of the contact structure and the topology of the corresponding toric 
symplectic filling.

\begin{thm} \label{thm:main3}
Twice the mean Euler characteristic of a Gorenstein toric contact manifold is equal to the
Euler characteristic of any crepant toric symplectic filling.
\end{thm}
\noindent In Section~\ref{s:resolutions} we illustrate this theorem with some families of examples 
of crepant toric symplectic fillings of Gorenstein toric contact manifolds. Note that in dimensions
higher than $5$ there are Gorenstein toric contact manifolds with no crepant toric symplectic filling.
In fact, as we show at the end of Section~\ref{s:resolutions}, the real projective spaces $(\rp^{4n+3}, 
\xi_{\rm std})$, $n\in\N$, are examples of that. 

\begin{remark}
As far as we know, it is unknown whether or not $(\rp^{4n+3}, \xi_{\rm std})$, $n\in\N$,
have (necessarily non-toric) symplectic fillings with zero first Chern class.
\end{remark}

\begin{remark}
Theorem~\ref{thm:main3} is not true if the filling is not toric. Indeed, as discussed in Example \ref{ex:S2xS3}, there
exists a family of toric contact structure $\xi_p$ on $S^2 \times S^3$, with $p \in \N$, such that $\xi_1$ can be
identified with the standard contact structure on the unit cosphere bundle of $S^3$ (see Section \ref{s:dim5}). 
The mean Euler characteristic of $\xi_1$ is equal to one, but clearly the Euler characteristic of the filling given 
by the unit disk bundle in $T^*S^3$ vanishes. As in Remark \ref{rem:apps}, this reflects the rigid nature of toric 
structures.
\end{remark}

Cho-Futaki-Ono show in~\cite{CFO}  that there exists an infinite family of inequivalent
toric Sasaki-Einstein metrics on $\#_{k} S^2 \times S^3$ for each $k\in\N$.  The invariant
they use to prove their inequivalence is precisely the volume of the corresponding toric
diagrams. It then follows from Theorem~\ref{thm:main2} and Remark~\ref{rem:invariance1} that the underlying 
Gorenstein toric contact structures are also distinct \emph{as contact structures}. Hence we have the following
result.
\begin{cor} \label{cor:app3}
For each $k\in\N$, there are infinitely many non-equivalent contact structures on $\#_{k} S^2 \times S^3$
in the unique homotopy class of almost contact strictures determined by the vanishing of the first Chern class. 
These contact structures are toric and can be distinguished by their mean Euler characteristic.
\end{cor}
\begin{rem} \label{rem:app3}
In the case $k=1$, i.e. $S^2 \times S^3$, these contact structures were first considered in~\cite{GW1}, 
underlying new Sasaki-Einstein metrics, and proved to have non-isomorphic cylindrical contact homology
in~\cite{AM}.
\end{rem}
\noindent In Section~\ref{s:dim5} we give one explicit description for toric diagrams associated to these 
Gorenstein toric contact structures, having also in mind the following natural question. 
\begin{question} \label{q:app3}
For each $k\in\N$, what is the minimal mean Euler characteristic of a Gorenstein toric contact structure
on $\#_{k} S^2 \times S^3$ ?
\end{question}
\noindent Lerman showed in~\cite{Le2} that the second homotopy group of any good toric contact manifold of
dimension $2n+1$ is a free abelian group of rank equal to $d-n-1$, where $d$ is the number of facets of
the corresponding good moment cone $C\subset\R^{n+1}$. It is then natural to reformulate and generalize
Question~\ref{q:app3} in the following purely combinatorial terms.
\begin{question} \label{q:appn}
Given $d,n\in\N$ with $d\geq n+1$, what is the minimal volume of a toric diagram $D\subset\R^n$ with
$d$ vertices?
\end{question}
\noindent When $n=2$ and $3\leq d \leq 16$ the answer is known (see~\cite{C}, where you can also find the
answer for a few higher values of $d$). Moreover, in this dimension it is possible to obtain a general bound 
using
\begin{itemize}
\item[1)] Pick's formula for the volume of a lattice polygon in $\R^2$ with $g$ interior lattice
points and $b$ lattice points on the boundary:
\[
\vol = g + \frac{b}{2} -1\,;
\] 
\item[2)] Coleman's conjecture, proved in~\cite{KO} (see also~\cite{C}), stating that for such a lattice $d$-gon
one has
\[
b \leq 2g + 10 - d \,.
\]
\end{itemize}
Since $d=b$ for toric diagrams $D\subset\R^2$, we get the inequality
\[
\vol (D) \geq \frac{3(d-4)}{2}\,.
\]
This is far from optimal, e.g. gives $18$ as the lower bound for the volume of a $16$-gon toric diagram while
its minimal volume is known to be $59$ (within the family of $16$-gons described in Section~\ref{s:dim5} the 
minimal volume is $63$). In any case, it can be combined with Theorem~\ref{thm:main2} to give a
(very) partial answer to Question~\ref{q:app3}.
\begin{cor} \label{cor:app4}
Let $(M,\xi)$ be a Gorenstein toric $5$-manifold. Then
\[
\Chi (M,\xi) \geq \frac{3}{2} (\rank(\pi_2 (M)) - 1)\,.
\]
In particular,
\[
\Chi (\#_{k} S^2 \times S^3,\xi) \geq \frac{3}{2} (k - 1)\,.
\]
\end{cor}

\subsection*{Acknowledgements} 
We are grateful to Leonor Godinho for pointing our attention to~\cite{C} and for helping us identify
real projective spaces $(\rp^{4n+3}, \xi_{\rm std})$, $n\in\N$, as examples of Gorenstein toric 
contact manifolds with no crepant toric symplectic filling (see the end of Section~\ref{s:resolutions}). 
During the 2015-16 academic year, Miguel Abreu benefited from the hospitality of the following institutions 
in addition to his own: Institut Mittag-Leffler (Sweden), Instituto de Matem\'atica Pura e Aplicada (Brazil) 
and Universidade Federal do Rio de Janeiro (Brazil).

\section{Gorenstein toric contact manifolds} 
\label{s:toric}

In this section we provide the necessary information on Gorenstein toric contact manifolds. For further details 
we refer the interested reader to~\cite{Le1} and~\cite{AM}.

\subsection{Toric symplectic cones}
\label{ss:cones}

Via symplectization, there is a $1$-$1$ correspondence between co-oriented contact manifolds
and symplectic cones, i.e. triples $(W,\om,X)$ where $(W,\om)$ is a connected symplectic manifold
and $X$ is a vector field, the Liouville vector field, generating a proper $\R$-action
$\rho_t:W\to W$, $t\in\R$, such that $\rho_t^\ast (\om) = e^{t} \om$. A closed symplectic cone is a 
symplectic cone $(W,\om,X)$ for which the corresponding contact manifold $M = W/\R$ is closed.

A toric contact manifold is a contact manifold of dimension $2n+1$ equipped with an effective Hamiltonian
action of the standard torus of dimension $n+1$: $\T^{n+1} = \R^{n+1} / 2\pi\Z^{n+1}$. Also via symplectization,
toric contact manifolds are in $1$-$1$ correspondence with toric symplectic cones, i.e. symplectic cones
$(W,\om,X)$ of dimension $2(n+1)$ equipped with an effective $X$-preserving Hamiltonian $\T^{n+1}$-action,
with moment map $\mu : W \to \R^{n+1}$ such that $\mu (\rho_t (w)) = e^{t} \mu (w)$, for all $w\in W$ and $t\in\R$.
Its moment cone is defined to be $C:= \mu(W) \cup \{ 0\} \subset \R^{n+1}$.

A toric contact manifold is {\it good} if its toric symplectic cone has a moment cone with the following properties.
\begin{definition} \label{def:good}
A cone $C\subset\R^{n+1}$ is \emph{good} if it is strictly convex and there exists a minimal set 
of primitive vectors $\nu_1, \ldots, \nu_d \in \Z^{n+1}$, with 
$d\geq n+1$, such that
\begin{itemize}
\item[(i)] $C = \bigcap_{j=1}^d \{x\in\R^{n+1}\mid 
\ell_j (x) := \langle x, \nu_j \rangle \geq 0\}$.
\item[(ii)] Any codimension-$k$ face of $C$, $1\leq k\leq n$, 
is the intersection of exactly $k$ facets whose set of normals can be 
completed to an integral basis of $\Z^{n+1}$.
\end{itemize}
The primitive vectors $\nu_1, \ldots, \nu_d \in \Z^{n+1}$ are called the defining normals of the good cone $C\subset\R^{n+1}$.
\end{definition}
The analogue for good toric contact manifolds of Delzant's classification theorem for closed toric
symplectic manifolds is the following result (see~\cite{Le1}).
\begin{theorem} \label{thm:good}
For each good cone $C\subset\R^{n+1}$ there exists a unique closed toric symplectic cone
$(W_C, \om_C, X_C, \mu_C)$ with moment cone $C$.
\end{theorem}
The existence part of this theorem follows from an explicit symplectic reduction of the standard euclidean
symplectic cone $(\R^{2d}\setminus\{0\}, \omega_{\rm st}, X_{\rm st})$, where $d$ is the number of defining normals of the
good cone $C\subset\R^{n+1}$, with respect to the action of a subgroup $K\subset\T^d$ induced by the standard
action of $\T^d$ on $\R^{2d}\setminus\{0\} \cong \C^d \setminus\{0\}$. More precisely,
\begin{equation} \label{eq:defK}
K := \left\{[y]\in\T^d \mid \sum_{j=1}^d y_j \nu_j \in 2\pi\Z^{n+1}   \right\}\,,
\end{equation}
where $\nu_1, \ldots, \nu_d \in \Z^{n+1}$ are the defining normals of $C$.

One source for examples of good toric contact manifolds is the prequantization construction over integral closed toric 
symplectic manifolds, i.e. $(M,\xi)$ with $M$ given by the $S^1$-bundle over $(B,\omega)$ with Chern class $[\omega]/2\pi$ and $\xi$ 
being the horizontal distribution of a connection with curvature $\omega$. The corresponding good cones have the form
\[
C:= \left\{z(x,1)\in\R^{n}\times\R\mid x\in P\,,\ z\geq 0\right\}
\subset\R^{n+1}
\]
where $P\subset\R^n$ is a Delzant polytope with vertices in the integer lattice $\Z^n\subset\R^n$.
Note that if
\[
P = \bigcap_{j=1}^d \{x\in\R^{n}\mid \langle x, v_j \rangle + \lambda_j \geq 0\}\,,
\]
with integral $\lambda_1, \ldots, \lambda_d \in \Z$ and primitive $v_1,\ldots, v_d \in \Z^n$, then
the defining normals of $C \subset \R^{n+1}$ are
\[
\nu_j = (v_j, \lambda_j)\,,\ j=1, \ldots, d\,.
\]

\subsection{First Chern class and toric diagrams}
\label{ss:c1diagram}

The Chern classes of a co-oriented contact manifold can be canonically identified with 
the Chern classes of the tangent bundle of the associated symplectic cone.
The vanishing of the first Chern class for good toric symplectic cones can be characterized
in several equivalent ways. Proposition 2.16 in~\cite{AM} does it in terms of the group
$K$ defined by~(\ref{eq:defK}), more precisely in terms of the characters that define its
representation in $\C^d$. The following proposition gives a characterization in terms of
the moment cone that will be more useful for our purposes and is commonly used in toric
Algebraic Geometry (see, e.g., \S $4$ of~\cite{BD}).
\begin{prop} \label{prop:c_1}
Let $(W_C, \omega_C,X_C)$ be a good toric symplectic cone.
Let $\nu_1,\ldots,\nu_d \in \Z^{n+1}$ be the defining normals of the corresponding moment cone 
$C\in\R^{n+1}$. Then $c_1 (TW_C) = 0$ if and only if there exists $\nu^\ast \in (\Z^{n+1})^\ast$ 
such that
\[
\nu^\ast (\nu_j) = 1\,,\ \forall\ j=1,\ldots,d\,.
\]
\end{prop}
By an appropriate change of basis of the torus $\T^{n+1}$, i.e. an appropriate $SL(n+1,\Z)$
transformation of $\R^{n+1}$, this implies the following.
\begin{cor} \label{cor:c_1}
Let $(W_C, \omega_C,X_C)$ be a good toric symplectic cone with $c_1 (TW_C) = 0$. Then there 
exists an integral basis of $\T^{n+1}$ for which the defining normals $\nu_1,\ldots,\nu_d \in \Z^{n+1}$ 
of the corresponding moment cone $C\subset\R^{n+1}$ are of the form
\[
\nu_j = (v_j, 1)\,,\ v_j \in \Z^n\,,\ j=1,\ldots,d\,.
\]
\end{cor}

The next definition and theorem are then the natural analogues for Gorenstein toric contact
manifolds of Definition~\ref{def:good} and Theorem~\ref{thm:good}.
\begin{defn} \label{def:diagram}
A \emph{toric diagram} $D\subset\R^n$ is an integral simplicial polytope with
all of its facets $\Aff(n,\Z)$-equivalent to $\conv(e_1, \ldots, e_n)$, where
$\{e_1,\ldots,e_n\}$ is the canonical basis of $\R^n$.
\end{defn}
\begin{rem} \label{rem:diagram}
The group $\Aff(n,\Z)$ of integral affine transformations of $\R^n$ can be naturally identified
with the elements of $SL(n+1,\Z)$ that preserve the hyperplane $\left\{ (v,1)\mid v\in\R^n \right\} \subset\R^{n+1}$.
\end{rem}
\begin{thm} \label{thm:diagram}
For each toric diagram $D\subset\R^n$ there exists a unique Gorenstein toric contact
manifold $(M_D, \xi_D)$ of dimension $2n+1$.
\end{thm}

Here are some examples of toric diagrams and corresponding Gorenstein toric contact manifolds.
\begin{example} \label{ex:lens}
The toric diagram $D = \conv (0, p) \subset \R$, with $p\in\N$, gives $(M_D, \xi_D) \cong (L(p,p-1), \xi_{\rm std})$, i.e.
a $3$-dimensional lens space with contact structure induced from the standard one on $S^3 = L(1,0)$. Figure~\ref{fig1}
a) shows the toric diagram of $L(3,2)$.
\end{example}
\begin{example} \label{ex:standardsphere}
The toric diagram $D = \conv (e_1, \ldots, e_n, \mathbf{0})\subset\R^n$, where $\{e_1, \ldots,e_n\}$ is the canonical
basis of $\R^n$ and $\mathbf{0}\in\R^n$ is the origin, gives $(M_D,\xi_D) \cong (S^{2n+1}, \xi_{\rm std})$, i.e.
the standard contact structure on the $(2n+1)$-sphere. In fact, $(S^{2n+1}, \xi_{\rm std})$ is the prequantization of
$(\CP^n, \omega_{\rm FS})$ with Delzant polytope
\[
P = \bigcap_{j=1}^{n+1} \{x\in\R^{n}\mid \langle x, v_j \rangle + \lambda_j \geq 0\}\,,
\]
where $v_j = e_j\,,\ \lambda_j = 0\,,\ j=1,\ldots,n$, and $v_{n+1} = - (e_1 + \cdots + e_n)\,,\ \lambda_{n+1} = 1$, and
so the corresponding good cone $C\subset\R^{n+1}$ has defining normals
\[
(e_j , 0)\,,\ j=1, \ldots, n\,,\ \text{and}\ (- (e_1 + \cdots + e_n),1)
\]
which are $SL(n+1, \Z)$-equivalent to
\[
(e_j , 1)\,,\ j=1, \ldots, n\,,\ \text{and}\ (\mathbf{0},1)
\]
(since both sets of normals form a $\Z$-basis of $\Z^{n+1}$). Figure~\ref{fig1} b) shows the toric diagram of
$(S^{5}, \xi_{\rm std})$.
\end{example}
\begin{example} \label{ex:quotientsphere}
The toric diagram $D = \conv (e_1, \ldots, e_n, -(e_1 + \cdots + e_n))\subset\R^n$ gives $(M_D, \xi_D) =$
prequantization of $(\CP^n, \omega = (n+1) \omega_{\rm FS} = 2\pi c_1 (\CP^n))$, i.e. a $\Z_{n+1}$-quotient
of $(S^{2n+1}, \xi_{\rm std})$. Figure~\ref{fig1} c) shows the toric diagram of
$(S^{5} / \Z_3, \xi_{\rm std})$.
\end{example}
\begin{example} \label{ex:monotone}
Any monotone Delzant polytope $P\subset\R^n$ with primitive normal vectors
$v_1, \ldots, v_d \in \Z^n \subset \R^n$ determines a toric diagram
$D := \conv(v_1,\ldots, v_d) \subset \R^n$. The corresponding Gorenstein toric contact manifold
$(M_D, \xi_D)$ is the prequantization of the monotone toric symplectic manifold $(B_P, \omega_P)$ 
determined by $P$, with $[\omega_P] = 2\pi c_1 (M_P)$.
\end{example}
\begin{example} \label{ex:S2xS3}
The toric diagram $D = \conv ((0,0), (1,0), (0,1), (p,p))\subset\R^2$, with $p\in\N$, gives
$(M_D,\xi_D) \cong (S^2\times S^3, \xi_{p})$, i.e. a family of contact structures on $S^2 \times S^3$.
This is the family of Remark~\ref{rem:app3}. Figure~\ref{fig1} d) shows the toric diagram of
$(S^2\times S^3, \xi_{3})$.
\end{example}

\begin{figure}[ht]
\includegraphics{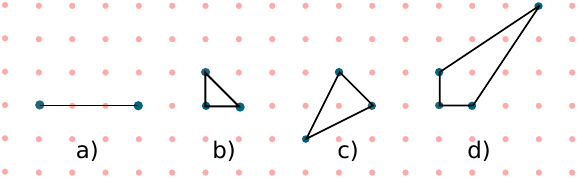}
\centering
\caption{Examples of toric diagrams.}
\label{fig1}
\end{figure}

\subsection{Fundamental group}
\label{ss:pi1}

As shown by Examples \ref{ex:lens} and \ref{ex:quotientsphere}, a Gorenstein toric contact
manifold $(M_D, \xi_D)$ can have nontrivial fundamental group. In fact, it follows from a result of
Lerman~\cite{Le2} that if $D = \conv (v_1, \ldots, v_d)$ then the fundamental group of $M_D$
is the finite abelian group
\[
\Z^{n+1}/\Nn\,,
\]
where $\Nn$ denotes the sublattice of $\Z^{n+1}$ generated by $\{\nu_1=(v_1,1), \ldots, \nu_d=(v_d,1)\}$
Hence, for simply connected Gorenstein toric contact manifolds we have that the $\Z$-span of the set of
defining normals $\{\nu_1, \ldots, \nu_d\}$ is the full integer lattice $\Z^{n+1}\subset\R^{n+1}$. 

\subsection{Normalized toric Reeb vectors}
\label{ss:reeb}

Let $(W,\omega, X)$ be a good toric symplectic cone of dimension $2(n+1)$, with
corresponding closed toric contact manifold $(M,\xi)$. Denote by $\Xx_X (W, \omega)$
the set of $X$-preserving symplectic vector fields on $W$ and by $\Xx (M,\xi)$
the corresponding set of contact vector fields on $M$. The $\T^{n+1}$-action
associates to every vector $\nu \in \R^{n+1}$ a vector field
$R_\nu \in \Xx_X (W,\omega) \cong \Xx (M, \xi)$. We will say that a
contact form $\alpha_\nu \in \Omega^1 (M,\xi)$ is \emph{toric} if
its Reeb vector field $R_{\alpha_\nu}$ satisfies
\[
R_{\alpha_\nu} = R_\nu \quad\text{for some $\nu\in\R^{n+1}$.}
\]
In this case we will say that $\nu\in\R^{n+1}$ is a \emph{Reeb vector}
and that $R_\nu$ is a \emph{toric Reeb vector field}.
The following proposition characterizes which $\nu\in\R^{n+1}$ are Reeb 
vectors of a toric contact form on $(M,\xi)$.

\begin{prop}[{\cite{MSY} or \cite[Proposition 2.19]{AM}}] \label{prop:sasaki}
Let $\nu_1, \ldots, \nu_d \in \R^{n+1}$ be the defining normals of the good moment cone 
$C\in\R^{n+1}$ associated with $(W,\omega,X)$ and $(M,\xi)$. The vector field 
$R_\nu \in \Xx_X (W,\omega) \cong \Xx(M,\xi)$ is the Reeb vector field of a toric contact form 
$\alpha_\nu \in \Omega^1 (M,\xi)$ if and only if
\[
\nu = \sum_{j=1}^d a_j \nu_j \quad\text{with $a_j\in\R^+$ for all
$j=1, \ldots, d$.}
\]
\end{prop}

This motivates the following definition and corollary for toric Reeb vectors on Gorenstein
toric contact manifolds.
\begin{defn} \label{def:reeb}
A \emph{normalized} toric Reeb vector is a toric Reeb vector $\nu\in\R^{n+1}$ of the form
\[
\nu = \sum_{j=1}^d a_j \nu_j \quad\text{with $a_j\in\R^+$ for all
$j=1, \ldots, d$, and}\quad \sum_{j=1}^d a_j = 1\,.
\]
\end{defn}
\begin{cor} \label{cor:reeb}
The interior of a toric diagram $D \subset \R^n$ parametrizes the set of normalized toric 
Reeb vectors on the Gorenstein toric contact manifold $(M_D, \xi_D)$.
\end{cor}
\begin{proof}
If $D = \conv(v_1, \ldots, v_d)$ then $\nu_j = (v_j ,1)$, $j=1, \ldots,d$, and any normalized toric
Reeb vector is of the form
\[
\nu = (v,1) \quad\text{with}\quad v = \sum_{j=1}^d a_j v_j\,,\ a_j\in\R^+\,,\ j=1, \ldots, d\,,\ \text{and}\ 
\sum_{j=1}^d a_j = 1\,.
\]
\end{proof}

\subsection{Parity of contact homology degree of closed toric Reeb orbits}
\label{ss:parity}

Let $(W_C, \omega_C,X_C)$ be a good toric symplectic cone, not necessarily Gorenstein, with 
$\nu_1,\ldots,\nu_d \in \Z^{n+1}$ being the defining normals of the corresponding good moment cone 
$C\in\R^{n+1}$. Denote by $(M_C, \xi_C)$ the corresponding closed toric contact manifold and consider 
a toric Reeb vector field $R_\nu \in \Xx (M_C, \xi_C)$ determined by the toric Reeb vector
\[
\nu = \sum_{j=1}^d a_j \nu_j\,,\ a_j\in\R^+\,,\ j=1, \ldots, d\,.
\]
The toric Reeb flow of $R_\nu$ on $(M_C,\xi_C)$ has at least $m$ simple closed orbits $\gamma_1, \ldots,\gamma_m$, 
corresponding to the $m$ edges $E_1,\ldots,E_m$ of the cone $C$, i.e. one simple closed toric $R_\nu$-orbit for each 
$S^1$-orbit of the $\T^{n+1}$-action on $(M_C,\xi_C)$. When $R_\nu$ is non-degenerate, i.e. when 
$\nu\in\R^{n+1} \equiv \rm{Lie} \, (\T^{n+1})$ generates a dense $1$-parameter subgroup of $\T^{n+1}$, then 
$\gamma_1, \ldots,\gamma_m$ are its only simple closed orbits.

\begin{prop} \label{prop:parity}
When $R_\nu$ is non-degenerate the contact homology degree of $\gamma_j^N$, for any $j=1, \ldots,m$, and any iterate 
$N\in\N$, is even.
\end{prop}
\begin{proof}
The parity of the contact homology degree is independent of the trivialization of the contact structure used to define it.
In fact, the contact homology degree of a non-degenerate closed Reeb orbit is given by the Conley-Zehnder index of its
linearization plus $n-2$, and different trivializations can only change the Conley-Zehnder index by twice the value of the 
Maslov index of an identity based loop in the symplectic linear group $Sp(2n,\R)$. 

Hence, it is enough to show that the contact homology degree is even for some trivialization. We will prove that by showing 
that any $\gamma_j$, $j=1, \ldots,m$, has a $\T^{n+1}$-invariant neighborhood that is $\T^{n+1}$-equivariantly 
contactomorphic to a $\T^{n+1}$-invariant neighborhood of a simple closed orbit of a non-degenerate toric Reeb flow 
on the standard sphere $(S^{2n+1}, \xi_{\rm std})$. The result follows since the parity of the contact homology degree 
of any closed orbit of any such Reeb flow on the standard sphere is even. In fact, any such Reeb flow is just the Hamiltonian
flow on a irrational ellipsoid in $\R^{2(n+1)}$.

Let us denote by $E$ an arbitrary edge of the good moment cone $C\in\R^{n+1}$. Applying an appropriate $SL(n+1,\Z)$ 
change of basis we may assume that the normals to the $n$ facets of $C$ that contain $E$ are the first $n$ vectors of the 
canonical basis of $\R^{n+1}$, which we denote by $\{e_1, \ldots, e_{n+1}\}$, and that $E$ is generated by $e_{n+1}$. 
Write the Reeb vector $\nu$ as
\[
\nu = \sum_{j=1}^n b_j e_j + r  e_{n+1}\,, \ b_1, \ldots,b_n, r \in \R\,.
\]
Strict convexity of $C$ and Proposition~\ref{prop:sasaki} imply that $r > 0$. We can then pick 
\[
m_j \in \Z \quad\text{such that}\quad r_j := b_j - m_j r > 0\,,\ j=1,\ldots,n\,,
\]
and write the Reeb vector $\nu$ as
\[
\nu = \sum_{j=1}^n (r_j + m_j r) e_j + r  e_{n+1}\,, \ r_1, \ldots,r_n, r \in \R^+\,,\ m_1,\ldots,m_n\in\Z\,.
\]
Using again Proposition~\ref{prop:sasaki}, this implies that $\nu$ can be seen as a Reeb vector on the
good moment cone with $n+1$ facets defined by the following set of normals:
\[
\left\{e_1\,,\ \ldots\,,\ e_n\,,\ \sum_{j=1}^n m_j e_j + e_{n+1}\right\}\,.
\]
Since this set of vectors forms an integral basis of $\Z^{n+1}$, the corresponding toric contact manifold
is indeed the sphere $(S^{2n+1}, \xi_{\rm std})$.
\end{proof}

\begin{rem} \label{rem:parity}
Although the above argument is a very simple and direct way to prove Proposition~\ref{prop:parity}, it
does not give any further information regarding the actual value of the contact homology degree, which is
a well defined integer when the first Chern class of the contact structure vanishes. See~\cite[Section 5]{AM} 
for a way to compute that integer value for contractible non-degenerate closed toric Reeb orbits in Gorenstein 
toric contact manifolds, which we will use in Section~\ref{s:proof}.
\end{rem}

\section{Index of closed orbits and trivialization of the contact structure}
\label{s:trivialization}

In this section we will discuss the trivialization of the contact structure that we use to define the index 
of closed orbits. While this trivialization is standard for \emph{contractible} closed orbits, it depends 
on some choices when we consider not contractible periodic orbits. However, as will be explained below, 
it turns out that for Gorenstein toric contact manifolds there is a natural way to get this trivialization, which 
will be very handy for the computation of the mean index in the proof of Theorem \ref{thm:main}.

Let $(M^{2n+1},\xi)$ be a contact manifold, $\alpha$ a contact form supporting $\xi$ and $J$ an almost 
complex structure on $\xi$ compatible with $d\alpha|_\xi$. It is well known that the first Chern class 
$c_1(\xi) \in H^2(M;\Z)$ of the complex vector bundle $(\xi,J)$ does not depend on the choices of $\alpha$ 
and $J$. The same holds for the top complex exterior power $\Lambda_\C^n\xi$ up to complex bundle 
isomorphism.

Suppose that $c_1(\xi)=0$ so that $\Lambda_\C^n\xi$ is a trivial line bundle. Choose a trivialization 
$\tau: \Lambda_\C^n\xi \to M \times \C$ which corresponds to a choice of a non-vanishing section 
$\sec$ of $\Lambda_\C^n\xi$. The choice of this trivialization furnishes a unique way to symplectically 
trivialize the contact structure along periodic orbits of $\alpha$ up to homotopy. As a matter of fact, 
given a periodic orbit $\gamma: S^1 \to M$ of $\alpha$, let $\Phi: \gamma^*\xi \to S^1 \times \C^{n}$ 
be a trivialization of $\xi$ over $\gamma$ as a Hermitian vector bundle such that its highest complex 
exterior power coincides with $\tau$. This condition fixes the homotopy class of $\Phi$: 
given any other such trivialization $\Psi$ we have, for every $t \in S^1$, that $\Phi_t \circ \Psi_t^{-1}: \C^n \to \C^n$ 
has complex determinant equal to one and therefore the Maslov index of the symplectic path 
$t \mapsto \Phi_t \circ \Psi_t^{-1}$ vanishes, where $\Phi_t:=\pi_2 \circ \Phi|_{\gamma^*\xi(t)}$ and 
$\Psi_t:=\pi_2 \circ \Psi|_{\gamma^*\xi(t)}$ with $\pi_2: S^1 \times \C^n \to \C^n$ being the projection 
onto the second factor; cf. \cite{Esp,McL}. Notice that this trivialization is \emph{closed under iterations}, 
that is, the trivialization induced on $\gamma^N$ coincides, up to homotopy, with the $N$-th iterate of the 
trivialization over $\gamma$.

Thus, the choice of a non-vanishing section of $\Lambda_\C^n\xi$ furnishes a way to trivialize the contact structure over 
\emph{any} closed orbit $\gamma$ of $\alpha$. However, if $\gamma$ is \emph{contractible} there is a 
canonical way to trivialize $\gamma^*\xi$ unique up to homotopy. More precisely, consider a capping disk 
$\sigma$ of $\gamma$, that is, a smooth map $\sigma: D^2 \to M$, where $D^2$ is the two-dimensional disk, 
such that $\sigma|_{\partial D^2} = \gamma$. Choose a trivialization of $\sigma^*\xi$ and let 
$\Phi: \gamma^*\xi \to S^1 \times \R^{2n}$ be its restriction to the boundary. Since $D^2$ is contractible, 
the homotopy class of $\Phi$ does not depend on the choice of the trivialization of $\sigma^*\xi$. Moreover, 
the condition $c_1(\xi)=0$ ensures that the homotopy class of $\Phi$ does not depend on the choice of 
$\sigma$ as well.

The trivializations induced by a section $\sec$ of $\Lambda_\C^n\xi$ and a capping disk $\sigma$ coincide 
up to homotopy. Indeed, consider a trivialization $\Phi$ of $\sigma^*\xi$ as a Hermitian vector bundle and 
let $\sec^\Phi_\sigma$ be the section of $\sigma^*\Lambda_\C^n\xi$ induced by the top complex 
exterior power of $\Phi$. Since $D^2$ is contractible, $\sec^\Phi_\sigma$ is homotopic to $\sigma^*\sec$ 
and therefore the corresponding trivializations of $\gamma^*\xi$ are homotopic.

Given a non-vanishing section $\sec$ of $\Lambda_\C^n\xi$, one can define the Robbin-Salamon index $\rs(\gamma;\sec)$ of any closed orbit $\gamma$ in the usual way. More precisely, by the previous discussion $\sec$ induces a unique up to homotopy trivialization $\Phi: \gamma^*\xi \to S^1 \times \R^{2n}$. 
Using this trivialization, the linearized Reeb flow gives the symplectic path
\[
\Gamma(t) = \Phi_t \circ d\phi_\alpha^t(\gamma(0))|_\xi \circ \Phi_0^{-1},
\]
where $\phi^t_\alpha$ is the Reeb flow of $\alpha$. Then the Robbin-Salamon index $\rs(\gamma;\sec)$ 
is defined as the Robbin-Salamon index of $\Gamma$ \cite{RS}. It turns out that if $H^1(M;\Q)=0$ then 
this index does not depend on the choice of $\sec$ since every two such sections are homotopic; 
see \cite[Lemma 4.3]{McL}.

The Robbin-Salamon index coincides with the Conley-Zehnder index if $\gamma$ is non-degenerate 
and the mean index
\[
\Delta (\gamma;\sec) := \lim_{N\to\infty} \frac{\mu_{\rm RS} (\gamma^N;\sec)}{N}
\]
varies continuously with respect to the $C^2$-topology in the following sense: if $\alpha_j$ is a sequence 
of contact forms converging to $\alpha$ in the $C^2$-topology and $\gamma_j$ is a sequence of periodic 
orbits of $\alpha_j$ converging to $\gamma$ then $\Delta(\gamma_j;\sec) \xrightarrow{j\to\infty} \Delta(\gamma;\sec)$ 
\cite{SZ}.

Now, let $(M^{2n+1},\xi)$ be a Gorenstein toric contact manifold and $\alpha$ a toric contact form supporting 
$\xi$. As explained in the previous section, the symplectization $W$ of $M$ can be obtained by symplectic 
reduction of $\C^d$ by the action of a subtorus $K \subset \T^d$, where $d$ is the number of vertices of the 
corresponding toric diagram. Given a contractible closed orbit $\gamma$ of $\alpha$, it is possible to construct 
a Hamiltonian $H: \C^d \to \R$ whose Hamiltonian flow is unitary linear and has a closed orbit $\hat\gamma$ 
that lifts $\gamma$ and satisfies
\[
\rs(\hat\gamma)=\rs(\gamma),
\]
where the index in the left hand side can be computed using the canonical (constant) trivialization in $\C^d$ 
and the index in the right hand side is given by a trivialization using a capping disk; see \cite[Lemma 3.4]{AM}. 
This fact was used in \cite{AM} to compute the cylindrical contact homology of $M$ for \emph{contractible} 
closed orbits.

Choose a non-vanishing section $\sec$ of $\Lambda_\C^n\xi$. By the discussion above, $\sec$ furnishes a unique way, 
up to homotopy, to trivialize $\gamma^*\xi$ for any closed orbit $\gamma$ of $\alpha$. Since $\pi_1(M)$ 
is finite (see Subsection \ref{ss:pi1}) we have that $H^1(M;\Q)=0$ and therefore $\rs(\gamma;\sec)$ does 
not depend on the choice of $\sec$.

In this way, we are able to define the contact homology degree of \emph{every} closed orbit 
(contractible or not) and its computation for contractible orbits reduces to the computation 
of the index of a lifted periodic orbit of a suitable linear unitary Hamiltonian flow on $\C^d$ (this lift exists 
if and only if the closed orbit is contractible). Using the homogeneity of the mean index 
(i.e. $\Delta(\gamma^N)=N \Delta(\gamma)$ for every $N$) and the fact that $\pi_1(M)$ is finite, we can 
use this lift to compute the mean index of any closed orbit (because it is enough to compute the mean index 
of some contractible iterate of $\gamma$) and it will be crucial in the proof of 
Theorem \ref{thm:main} presented in the next section. (It is important here that, as noticed before, 
the trivialization induced by a section is closed under iterations and this property is essential to achieve 
the homogeneity of the mean index.)

\section{Proof of Theorem~\ref{thm:main}}
\label{s:proof}

Given a toric diagram $D = \conv (v_1, \ldots, v_d) \subset \R^n$ and corresponding Gorenstein 
toric contact manifold $(M_D, \xi_D)$, the proof of Theorem~\ref{thm:main} consists of the following
simple steps:
\begin{itemize}
\item[1)] Use the method developed in~\cite[Section 5]{AM} to compute the mean index of 
any closed simple Reeb orbit of any normalized non-degenerate toric Reeb vector field.
\item[2)] Compute the normalized volume of any simplicial pyramid with a normalized toric 
Reeb vector as vertex and a facet of $D$ as base.
\item[3)] Check that the values obtained in 1) and 2) are the same when the orbit of 1)
corresponds to the facet of 2).
\item[4)] Use the continuity of the mean index to conclude the result for possibly degenerate toric contact forms.
\end{itemize}

Consider a toric Reeb vector field $R_\nu \in \Xx (M_D, \xi_D)$ determined by the normalized toric 
Reeb vector
\[
\nu = (v,1) \quad\text{with}\quad v = \sum_{j=1}^d a_j v_j\,,\ a_j\in\R^+\,,\ j=1, \ldots, d\,,\ \text{and}\ 
\sum_{j=1}^d a_j = 1\,.
\]
By a small abuse of notation, we will also write
\[
R_\nu = \sum_{j=1}^d a_j \nu_j \,,
\]
where $\nu_j = (v_j,1)$, $j=1,\ldots,n$, are the defining normals of the associated good moment
cone $C\subset\R^n$.
Making a small perturbation of $\nu$ if necessary, we can assume that
\[
\text{the $1$-parameter subgroup generated by $R_\nu$ is dense in $\T^{n+1}$,}
\]
which means that if $v = (r_1, \ldots,  r_n)$ then the $r_j$'s are irrational and $\Q$-independent.
This is equivalent to the corresponding toric contact form being non-degenerate. 
In fact, as already pointed out in Subsection~\ref{ss:parity},
the toric Reeb flow of $R_\nu$ on $(M_D,\xi_D)$ has exactly $m$ simple closed
orbits $\gamma_1, \ldots,\gamma_m$, all non-degenerate, corresponding to the $m$ edges
$E_1,\ldots,E_m$ of the cone $C$, i.e. one non-degenerate closed simple toric $R_\nu$-orbit
for each $S^1$-orbit of the $\T^{n+1}$-action on $(M_D,\xi_D)$. Equivalently, there is
\[
\text{one non-degenerate closed simple toric $R_\nu$-orbit for each facet of the toric diagram $D$.}
\]

Let $F$ denote one of the facets of $D$, necessarily a simplex, and assume without loss of generality
that its vertices are $v_1, \ldots, v_n \in \Z^n$. Let $\eta\in\Z^n$ be such that
\[
\{ \nu_1 = (v_1,1), \ldots, \nu_n = (v_n,1), (\eta, 1)\} \ \text{is a $\Z$-basis of $\Z^{n+1}$.}
\]
Then $R_\nu$ can be uniquely written as
\[
R_\nu = \sum_{j=1}^n b_j \nu_j + b (\eta,1)\,,\ \text{with}\ b_1,\ldots, b_n\in \R\ \text{and}\ 
b = 1 - \sum_{j=1}^n b_j \ne 0\,.
\]
Let $k\in\N$ be the order in the fundamental group of $M_D$ of the non-degenerate closed simple toric 
$R_\nu$-orbit $\gamma$ determined by $F$ (that is, $k$ is the smallest positive integer such that $\gamma^k$ is contractible). Then the Conley-Zehnder index of $\gamma^{kN}$, for any 
$N\in\N$, is given by (see~\cite[Section 5]{AM})
\[
\mu_{\rm CZ} (\gamma^{kN}) = 2 \left( \sum_{j=1}^n \left\lfloor kN \frac{b_j}{|b|}\right\rfloor
 + kN \frac{b}{|b|}\right) + n\,.
\]
Here we are taking a trivialization of $\xi_D$ induced by a non-vanishing section of $\Lambda^n_\C\xi_D$ as discussed in Section \ref{s:trivialization}. Hence, for the mean index of $\gamma^k$ we have that
\[
\Delta (\gamma^k) = \lim_{N\to\infty} \frac{\mu_{\rm CZ} (\gamma^{kN})}{N} =
2 \left( \sum_{j=1}^n  k\frac{b_j}{|b|} + k \frac{b}{|b|} \right) = \frac{2k}{|b|}\,,
\]
which implies that
\[
\Delta (\gamma) = \frac{2}{|b|},
\]
where we are using the homogeneity of the mean index (see Section \ref{s:trivialization}).
\begin{rem} \label{rem:positive}
This shows that the mean index of any closed orbit of any normalized non-degenerate toric Reeb vector field
is positive. This implies that $\inf \{j\in\Z \mid cb_j (M,\alpha) \ne 0\}$ is finite for
any non-degenerate toric contact form $\alpha$, as claimed in Remark~\ref{rem:meanEuler}.
\end{rem}

Let us now prove that
\[
\frac{1}{\Delta (\gamma)} = \frac{n! \vol ( S_{v})}{2}\,,\ \text{i.e.}\ n! \vol ( S_{v}) = |b|\,,
\]
where $S_{v} = \conv (v, v_1, \ldots, v_n)$. Using $\Aff(n,\Z)$-invariance, we may assume without 
loss of generality that
\[
\eta = 0 \ \text{and}\ v_1 = e_1, \ldots, v_n = e_n\,,\ \text{where}\ \{e_1, \ldots, e_n\}\ \text{is the standard 
basis of $\R^n$.}
\]
In this case we have that $v= (b_1, \ldots, b_n)$ and so
\[
n! \vol (S_v) = |\det (I_n - A_v)| \,,
\]
where $I_n$ is the $(n\times n)$ identity matrix and $A_v$  is the $(n\times n)$ matrix with all of its 
columns given by $v \in\R^n$. A simple determinant computation then gives 
\[
|\det (I_n - A_v)| = |1-\sum_{j=1}^n b_j | = |b|\,.
\]
Finally, using the continuity of the mean index (see Section \ref{s:trivialization}) we easily conclude the proof of the Theorem when $R_\nu$ is degenerate.

\section{Examples of crepant toric symplectic fillings}
\label{s:resolutions}

In this section we illustrate Theorem~\ref{thm:main3} with some families of examples of crepant toric symplectic 
fillings of Gorenstein toric contact manifolds and show that the real projective spaces $(\rp^{4n+3}, \xi_{\rm std})$, 
$n\in\N$, are examples of Gorenstein toric contact manifolds with no crepant toric symplectic filling. All these
correspond to examples of isolated toric Gorenstein singularities well known in Algebraic Geometry. To check that 
the Euler characteristic of a filling is indeed equal to the normalized volume of the toric diagram, recall the following 
well known fact: the Euler characteristic of a toric symplectic manifold is equal to the number of vertices of the 
corresponding moment map image.

The first family of examples is given by the $3$-dimensional lens spaces $(L(p,p-1), \xi_{\rm std}), p\in\N$ 
(cf. Example~\ref{ex:lens}). The corresponding toric diagrams are $D = \conv (0, p) \subset \R$ with normalized
volume equal to $p$ and the moment map image of the corresponding symplectizations is a good cone in
$\R^2$ with defining normals $(0,1)$ and $(p,1)$. The moment map image of a crepant toric symplectic filling is
a convex polyhedral set in $\R^2$ with primitive interior normals to its edges given by $(0,1)$, $(1,1)$, \dots, $(p-1,1)$,
$(p,1)$. See Figure~\ref{fig6} for the case $p=3$. The Euler characteristic of this filling is indeed equal to $p$. 

\begin{figure}[ht]
\includegraphics{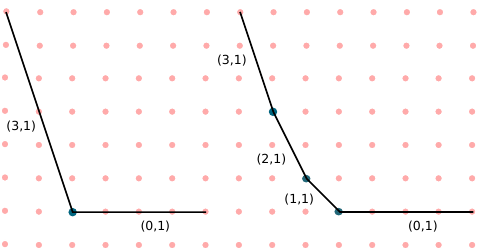}
\centering
\caption{Moment map images of the symplectization of $(L(3,2), \xi_{\rm std})$ on the left and of its crepant toric symplectic
filling on the right. Each label indicates the primitive interior normal to the corresponding edge.}
\label{fig6}
\end{figure}

The second family of examples is given by the prequantization of monotone toric symplectic manifolds, as described
in Example~\ref{ex:monotone}. Here a crepant toric symplectic filling is obtained by blowing-up the Gorenstein toric
isolated singularity. In the moment map image, this corresponds to cutting the good cone at level one with a horizontal
facet. See Figure~\ref{fig7} for the case where the basis of the prequantization is 
$(\CP^2, \omega = 3 \omega_{\rm FS} = 2\pi c_1 (\CP^2))$ (cf. Example~\ref{ex:quotientsphere}).
In all these prequantization examples, one can use an elementary simplicial subdivision argument  to show that the 
normalized volume of the toric diagram is equal to the number of its facets. That is also the number of vertices 
of the moment polytope of the basis, which is the Euler characteristic of the basis and of the crepant toric symplectic filling.

\begin{figure}[ht]
\includegraphics{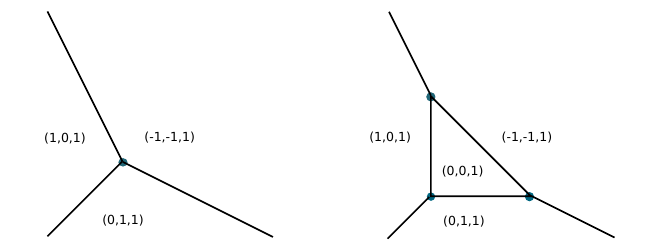}
\centering
\caption{Moment map images of the symplectization of $(S^5/\Z_3, \xi_{\rm std})$ on the left and of its crepant toric symplectic
filling on the right. The viewer is on the inside of the moment map images and each label indicates the primitive interior normal 
to the corresponding facet.}
\label{fig7}
\end{figure}

The third and final family of examples is given by the Gorenstein toric contact structures of Example~\ref{ex:S2xS3}, i.e.
$(S^2\times S^3, \xi_{p}),\, p\in\N$. The corresponding toric diagrams are $D = \conv ((0,0), (1,0), (0,1), (p,p))\subset\R^2$ 
with normalized volume equal to $2p$ and the moment map image of the corresponding symplectizations is a good cone in
$\R^3$ with defining normals $(0,0,1), (1,0,1),(0,1,1)$ and $(p,p,1)$. The moment map image of a crepant toric symplectic filling is
a convex polyhedral set in $\R^3$ with primitive interior normals to its edges given by $(0,0,1), (1,0,1),(0,1,1)$, $(1,1,1)$, \dots, $(p-1,p-1,1)$,
$(p,p,1)$. See Figure~\ref{fig9} for the case $p=3$. The Euler characteristic of this filling is indeed equal to $2p$. 

\begin{figure}[ht]
\includegraphics{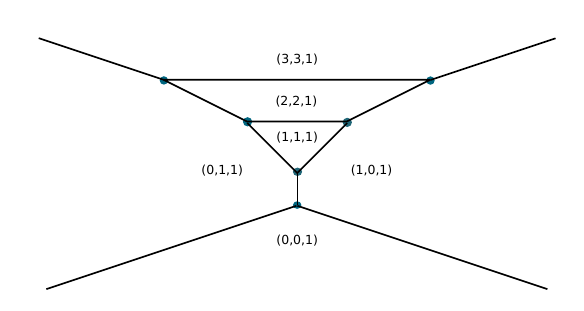}
\centering
\caption{Moment map image of a crepant toric symplectic filling of $(S^2\times S^3, \xi_3)$.
The viewer is on the inside of the moment map image and each label indicates the primitive interior normal 
to the corresponding facet.}
\label{fig9}
\end{figure}

In Algebraic Geometry, the crepant toric resolutions of isolated toric Gorenstein singularities are obtained via
regular simplicial subdivisions of the corresponding toric diagrams. The above examples correspond to obvious
regular subdivisions of their toric diagrams. Figure~\ref{fig8} illustrates that correspondence for the 
crepant toric symplectic fillings of Figures~\ref{fig7} and~\ref{fig9}. Note that the interior integral points determine 
the defining normals to the facets that are added to the moment cone for the construction of the crepant toric symplectic 
filling.

\begin{figure}[ht]
\includegraphics{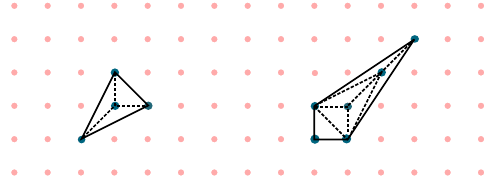}
\centering
\caption{Regular simplicial subdivisions of the toric diagrams of $(S^5/\Z_3, \xi_{\rm std})$ (left) and $(S^2\times S^3, \xi_3)$ (right),
corresponding to the crepant toric symplectic fillings of Figures~\ref{fig7} and~\ref{fig9}.}
\label{fig8}
\end{figure}

It is well known and easy to see that every toric diagram in $\R$ and $\R^2$ has a regular simplicial subdivision. The
associated crepant toric resolution of the corresponding isolated toric Gorenstein singularity can then be used to show
that all Gorenstein toric contact manifolds of dimensions $3$ and $5$ have crepant toric symplectic fillings. That is no
longer the case in higher dimensions. In fact, real projective spaces $(\rp^{4n+3}, \xi_{\rm std})$, $n\in\N$, already provide
examples of Gorenstein toric contact manifolds with no crepant toric symplectic filling. Although this is just a 
consequence of the fact that the origin in $(\C^{2(n+1)} / \pm 1)$, $n\in\N$, is an example of an isolated Gorenstein 
singularity with no crepant resolution, let us give an argument in the spirit of this paper, i.e. using toric diagrams.

Similarly to Example~\ref{ex:standardsphere}, $(\rp^{2n+1}, \xi_{\rm std})$ can be described as the prequantization of 
$(\CP^{n}, 2\omega_{\rm FS})$. The corresponding good moment cone $C\subset \R^{n+1} =\R^n \times \R$ has defining normals
\[
(e_j , 0)\,,\ j=1,\ldots,n \quad\text{and}\quad (-(e_1 + \cdots + e_n) , 2)\,,
\]
where $\{e_1, \ldots, e_n\}$ denotes the canonical basis of $\R^n$. Using Proposition~\ref{prop:c_1} one can show that
\[
c_1 (\rp^{2n+1}, \xi_{\rm std}) = 0 \Leftrightarrow n\ \text{odd.}
\]
Hence, any  $(\rp^{4n+3}, \xi_{\rm std})$, $n\in\N_0$, is a Gorenstein toric contact manifold. In fact, in these dimensions
the above set of normals is $SL(2(n+1),\Z)$-equivalent to
\[
(\mathbf{0},1)\,,\ (e_j , 0,1)\,,\ j=1, \ldots, 2n\,,\ \text{and}\ (e_1 + \cdots + e_{2n},2,1)\,,
\]
where $\mathbf{0} \in \R^{2n+1}$ and $\{e_1, \ldots, e_{2n}\}$ denotes the canonical basis of $\R^{2n}$.

When $n=0$ this is just $(\rp^3, \xi_{\rm std})$ seen as the lens space $(L(2,1), \xi_{\rm std})$ with 
toric diagram $D = \conv (0, 2) \subset \R$. As we have seen, it has a crepant toric symplectic filling. The
corresponding crepant resolution is obtained from the regular subdivision of its toric diagram that comes 
from having an interior integral point: $D = \conv (0, 2) = \conv (0,1) \cup \conv (1,2)$. 

In higher dimensions,
i.e. when $n>0$, the toric diagram
\[
D = \conv (\mathbf{0}, (e_1, 0), \ldots, (e_{2n},0), (e_1 + \cdots + e_{2n},2))\subset \R^{2n}\times \R = \R^{2n+1}
\]
has no interior integral points. Hence, the corresponding isolated toric Gorenstein singularity has no crepant
toric resolution, which implies that $(\rp^{4n+3}, \xi_{\rm std})$, $n\in\N$, has no crepant toric symplectic
filling. In other words, there are no normals one could use to define additional facets to the moment cone for
the construction of a toric symplectic filling that would be both smooth and with zero first Chern class.

\section{Simply-connected Gorenstein toric contact structures in dimension $5$}
\label{s:dim5}

In this section we give one explicit description for the families of contact structures that one needs
to construct in order to prove Corollary~\ref{cor:app3}. First let us recall that the diffeomorphism type
of a simply-connected Gorenstein toric $5$-manifold is completely determined by the number of
vertices of its toric diagram.
\begin{prop}[{\cite[Theorem 5.5]{O}}] \label{prop:diif5}
Let $(M,\xi)$ be a simply-connected Gorenstein toric contact $5$-manifold determined by a toric
diagram $D\subset\R^2$ with $d$ vertices. Then $M$ is diffeomorphic to $S^5$ when $d=3$ and to
\[
\#_{d-3} S^2 \times S^3 \quad\text{when}\quad d>3\,.
\]
\end{prop}

A good moment cone for a toric contact structure on $S^5$ has $3$ normals that form a $\Z$-basis
of $\Z^3$. Hence it is $SL(3,\Z)$ equivalent to the good moment cone with normals
\[
(0,0,1)\,,\ (1,0,1)\quad\text{and}\quad (0,1,1)\,,
\]
and gives the standard contact structure on $S^5$. Its toric diagram is the unimodular simplex with
vertices $(0,0)$, $(1,0)$ and $(0,1)$, shown in Figure~\ref{fig1} b). Its normalized volume is $1$ and 
we recover the well known fact that
\[
\Chi (S^5, \xi_{\rm st}) = \frac{1}{2}\,.
\]

A minimal volume toric diagram with $4$ vertices is the square with vertices $(0,0)$, $(1,0)$, $(0,1)$
and $(1,1)$. It determines the standard contact structure on $S^2\times S^3$, as the unit cosphere
bundle of $S^3$, with mean Euler characteristic equal to $1$. The family of $4$-gons with vertices
\[
(0,0)\,,\ (1,0)\,,\ (0,1)\quad\text{and}\quad (p,p)\,,\ p\in\N\,,
\]
gives a family of inequivalent contact structures $\xi_p$ on $S^2\times S^3$ with
\[
\Chi (S^2\times S^3, \xi_p) = p \,.
\]
See Example~\ref{ex:S2xS3} and Figure~\ref{fig1} d).

A minimal volume toric diagram with $5$ vertices is the $5$-gon with vertices $(-1,0)$, $(0,-1)$, $(1,0)$, $(0,1)$
and $(1,1)$. It determines a contact structure on $\#_ 2 S^2\times S^3$ with mean Euler characteristic equal 
to $5/2$. The family of $5$-gons with vertices
\[
(0,-1)\,,\ (-1,0)\,,\ (1,0)\,,\ (0,1)\quad\text{and}\quad (p,p)\,,\ p\in\N\,,
\]
gives a family of inequivalent contact structures $\xi_p$ on $\#_2 S^2\times S^3$ with
\[
\Chi (\#_2 S^2\times S^3, \xi_p) = \frac{3}{2} + p\,.
\]

A minimal volume toric diagram with $6$ vertices is the $6$-gon with vertices $(-1,-1)$, $(0,-1)$, $(-1,0)$, $(0,1)$,
$(1,0)$ and $(1,1)$ (see Figure~\ref{fig2} left). It determines a contact structure on $\#_ 3 S^2\times S^3$ with mean Euler characteristic equal 
to $3$. The family of $6$-gons with vertices
\[
(-1,-1)\,,\ (0,-1)\,,\ (-1,0)\,,\ (1,0)\,,\ (0,1)\quad\text{and}\quad (p,p)\,,\ p\in\N\,,
\]
gives a family of inequivalent contact structures $\xi_p$ on $\#_3 S^2\times S^3$ with
\[
\Chi (\#_3 S^2\times S^3, \xi_p) = 2 + p\,.
\]
See Figure~\ref{fig2} right for this toric diagram with $p=4$. Note that, as indicated in Figure~\ref{fig2}, clipping off the vertex $(-1,-1)$ in this family of $6$-gons gives the family of $5$-gons considered above.

\begin{figure}[ht]
\includegraphics{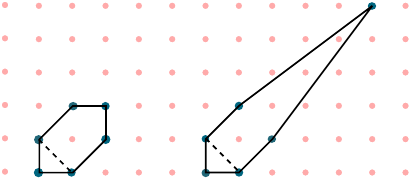}
\centering
\caption{Examples of toric diagrams with $6$ vertices. The one on the left has minimal volume.
The dashed segment indicates that clipping off the 
bottom left vertex gives examples of toric diagrams with $5$ vertices.}
\label{fig2}
\end{figure}

A minimal volume toric diagram with $8$ vertices is the $8$-gon with vertices $(0,0)$, $(1,0)$, $(0,3)$, $(1,3)$,
$(-1,1)$, $(-1,2)$, $(2,1)$ and $(2,2)$ (see Figure~\ref{fig3} left). It determines a contact structure on 
$\#_ 5 S^2\times S^3$ with mean Euler characteristic equal to $7$. The family of $8$-gons with vertices
\[
(0,0)\,,\ (1,0)\,,\ (0,3)\,,\ (1,3)\,,\ (-1,1)\,,\ (-1,2)\,,\ (1+p,1)\quad\text{and}\quad (1+p,2)\,,\ p\in\N\,,
\]
gives a family of inequivalent contact structures $\xi_p$ on $\#_5 S^2\times S^3$ with
\[
\Chi (\#_5 S^2\times S^3, \xi_p) = 5 + 2p\,.
\]
See Figure~\ref{fig3} right for this toric diagram with $p=4$. Note that, as indicated in Figure~\ref{fig3}, 
clipping off the vertex $(-1,1)$ in this family of $8$-gons gives a family of $7$-gons that determine a 
family of inequivalent contact structures $\xi_p$ on $\#_4 S^2\times S^3$. The $7$-gon on the left of 
Figure~\ref{fig3} also has minimal volume.

\begin{figure}[ht]
\includegraphics{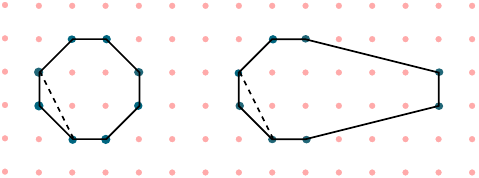}
\centering
\caption{Examples of toric diagrams with $8$ vertices. The one on the left has minimal volume.
The dashed segment indicates how to obtain
examples of toric diagrams with $7$ vertices by clipping off one vertex.}
\label{fig3}
\end{figure}

One possibility to systematically increase the number of vertices is the following (see Figure~\ref{fig4}).
For each $k\in\N$ consider the family of $(4k+4)$-gons, parametrized by $p\in\N$, with vertices
\[
(0,0)\,,\ (1,0)\,,\ (0, 2k+1)\,,\ (1, 2k+1)\,,
\]
\[
\left(-\frac{(k-j)(k+j+1)}{2}, k-j \right)\,,\ \left(-\frac{(k-j)(k+j+1)}{2}, k+j+1\right)\,,
\]
\[
\left(\frac{(k-j)(k+j+1)}{2} + p, k-j\right)\,,\ \left(\frac{(k-j)(k+j+1)}{2} + p, k+j+1\right)\,,\ j = 0, \ldots, k-1\,.
\]
It gives a family of inequivalent contact structures $\xi_p$ on $\#_{4k+1} S^2\times S^3$ with
\[
\Chi (\#_{4k+1} S^2\times S^3, \xi_p) = 2k+1 + \frac{2k(k+1)(2k+1)}{3} + (p-1)2k\,.
\]
Clipping for example the vertex $(0,0)$ in this family, one gets a family of $(4k+3)$-gons that determine 
a family of inequivalent contact structures $\xi_p$ on $\#_{4k} S^2\times S^3$ with
\[
\Chi (\#_{4k} S^2\times S^3, \xi_p) = 2k+\frac{1}{2} + \frac{2k(k+1)(2k+1)}{3} + (p-1)2k\,.
\]

\begin{figure}[ht]
\includegraphics{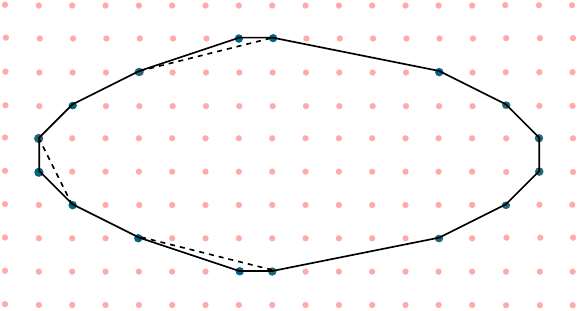}
\centering
\caption{Example of a toric diagram with $16$ vertices, corresponding to $k=p=3$. The dashed segments indicate 
how to obtain examples of toric diagrams with $15$, $14$ and $13$ vertices by clipping off one, two or three vertices.}
\label{fig4}
\end{figure}

For each $k>1$ we can clip one or two more vertices in this family of $(4k+4)$-gons, for example the 
vertices with coordinates $(0, 2k+1)$ and $(-k(k+1)/2, k)$, as shown in Figure~\ref{fig4} when $k=p=3$. This
gives families of $(4k+2)$-gons and $(4k+1)$-gons that determine families of inequivalent contact structures 
$\xi_p$ on $\#_{4k-1} S^2\times S^3$ and $\#_{4k-2} S^2\times S^3$ with
\[
\Chi (\#_{4k-1} S^2\times S^3, \xi_p) = 2k + \frac{2k(k+1)(2k+1)}{3} + (p-1)2k
\]
and
\[
\Chi (\#_{4k-2} S^2\times S^3, \xi_p) = 2k - \frac{1}{2} + \frac{2k(k+1)(2k+1)}{3} + (p-1)2k\,.
\]

When $k=p=1$ the corresponding $7$-gon and $8$-gon are the ones in Figure~\ref{fig3} left and 
have minimal volume. That is no longer the case when $k>1$.


For the record, a minimal volume toric diagram with $9$ vertices is the $9$-gon with vertices 
$(-2,-3)$, $(-3,-2)$, $(-1,-3)$, $(-3,-1)$, $(0,-2)$, $(-2,0)$, $(1,0)$, $(0,1)$ and $(1,1)$ (see Figure~\ref{fig5} left).
It determines a contact structure on $\#_ 6 S^2\times S^3$ with mean 
Euler characteristic equal to $21/2$. The family of $9$-gons with vertices
\[
(-2,-3),\ (-3,-2),\ (-1,-3),\ (-3,-1),\ (0,-2),\ (-2,0),\ (p,p-1),\ (p-1,p),\  (p,p),\ p\in\N,
\]
gives a family of inequivalent contact structures $\xi_p$ on $\#_6 S^2\times S^3$ with
\[
\Chi (\#_6 S^2\times S^3, \xi_p) = \frac{9}{2} + 3(1 + p)\,.
\]
See Figure~\ref{fig5} right for this toric diagram with $p=3$.

\begin{figure}[ht]
\includegraphics{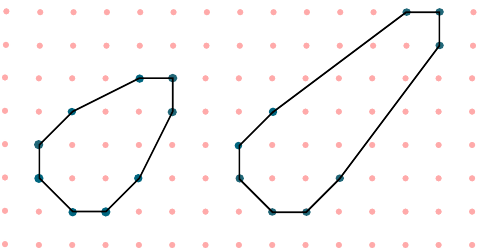}
\centering
\caption{Examples of toric diagrams with $9$ vertices. The one on the left has minimal volume.}
\label{fig5}
\end{figure}

\end{document}